\documentclass[reqno]{amsart}
\usepackage{hyperref}
\usepackage{amsmath, amsthm, amscd, amsfonts, amssymb}
\newtheorem{theorem}{Theorem}[section]
\newtheorem{lemma}[theorem]{Lemma}

\theoremstyle{definition}

\theoremstyle{remark}
\newtheorem{remark}[theorem]{Remark}

\begin{document}
\title[Some Inequalities for the Gamma Function]
{Generalizations of Some Inequalities for the $p$-Gamma, $q$-Gamma and $k$-Gamma  Functions}

\author[Kwara Nantomah, Edward Prempeh]
{Kwara Nantomah, Edward Prempeh}

\address{Kwara Nantomah \newline
Department of Mathematics, University for Development Studies, Navrongo Campus, P. O. Box 24, Navrongo, UE/R, Ghana.}
\email{mykwarasoft@yahoo.com, knantomah@uds.edu.gh}

\address{Edward Prempeh \newline
 Department of Mathematics, Kwame Nkrumah University of Science and Technology, Kumasi, Ghana.}
\email{eprempeh.cos@knust.edu.gh}


\maketitle

\setcounter{page}{172}
\begin{abstract}
In this paper, we present and prove some generalizations of some inequalities for  the $p$-Gamma, $q$-Gamma and $k$-Gamma functions. Our approach makes use of the series representations of the psi, $p$-psi, $q$-psi  and $k$-psi functions.
\end{abstract}

\section{Introduction}
We begin by recalling some basic definitions related to the Gamma function.\\

The Psi function, $\psi(t)$  is defined as,
\begin{equation}\label{eqn:psi}
\psi(t)=\frac{d}{dt}\ln(\Gamma(t))=\frac{\Gamma'(t)}{\Gamma(t)}, \qquad t>0.
\end{equation}

where $\Gamma(t)$ is the classical Euler's Gamma function defined by
\begin{equation}\label{eqn:gamma}
\Gamma(t)=\int_0^\infty e^{-x}x^{t-1}\,dx, \qquad t>0.
\end{equation}

The $p$-psi function, $\psi_p(t)$  is defined as,
\begin{equation}\label{eqn:p-psi}
\psi_p(t)=\frac{d}{dt}\ln(\Gamma_p(t))=\frac{\Gamma_p'(t)}{\Gamma_p(t)},\qquad t>0.
\end{equation}

where $\Gamma_p(t)$ is the $p$-Gamma function defined by (see \cite{Krasniqi-Merovci-2010}, \cite{Krasniqi-Mansour-Shabani-2010} )
\begin{equation}\label{eqn:p-gamma}
\Gamma_p(t)=\frac{p!p^t}{t(t+1) \dots (t+p)}=\frac{p^t}{t(1+\frac{t}{1}) \dots (1+\frac{t}{p})}, \quad p\in N,\quad t>0.
\end{equation}

The $q$-psi function, $\psi_q(t)$  is defined as,
\begin{equation}\label{eqn:q-psi}
\psi_q(t)=\frac{d}{dt}\ln(\Gamma_q(t))=\frac{\Gamma_q'(t)}{\Gamma_q(t)},\qquad t>0.
\end{equation}

where $\Gamma_q(t)$ is the $q$-Gamma function defined by (see \cite{Mansour-2008})
\begin{equation}\label{eqn:q-gamma}
\Gamma_q(t)=(1-q)^{1-t}\prod_{n=1}^{\infty}\frac{1-q^n}{1-q^{t+n}} , \quad q\in (0,1),\quad t>0.
\end{equation}

\noindent
Similarly, the $k$-psi function, $\psi_k(t)$ is  defined as follows. 
\begin{equation}\label{eqn:k-psi}
\psi_k(t)=\frac{d}{dt}\ln(\Gamma_k(t))=\frac{\Gamma_k'(t)}{\Gamma_k(t)},\quad t>0.
\end{equation}

\noindent
where $\Gamma_k(t)$ is the  $k$-Gamma function defined by (see \cite{Diaz2007}, \cite{Merovci2010} ) 
\begin{equation}\label{eqn:k-gamma}
\Gamma_k(t)=\int_0^\infty e^{-\frac{x^k}{k}}x^{t-1}\,dx, \quad k>0,\quad t>0.
\end{equation}

\noindent
In \cite{Krasniqi-Shabani-2010}, Krasniqi and Shabani proved the following results.
\begin{equation}\label{eqn:p-ineq1}
\frac{p^{-t}e^{-\gamma t}\Gamma(\alpha)}{\Gamma_p(\alpha)}< 
\frac{\Gamma(\alpha+t)}{\Gamma_p(\alpha+t)}<
\frac{p^{1-t}e^{\gamma(1-t)} \Gamma(\alpha+1)}{\Gamma_p(\alpha+1)}
\end{equation}
for $t\in(0,1)$, where  $\alpha$ is a positive real number such that $\alpha+t>1$.\\

\noindent
Also in \cite{Krasniqi-Mansour-Shabani-2010}, Krasniqi, Mansour and Shabani proved the following. 
\begin{equation}\label{eqn:q-ineq1}
\frac{(1-q)^{t}e^{-\gamma t} \Gamma(\alpha)}{\Gamma_q(\alpha)}< 
\frac{\Gamma(\alpha+t)}{\Gamma_q(\alpha+t)}<
 \frac{(1-q)^{t-1}e^{\gamma(1-t)}\Gamma(\alpha+1)}{\Gamma_q(\alpha+1)}
\end{equation}
for $t\in(0,1)$, where  $\alpha$ is a positive real number such that $\alpha+t>1$ and $q\in(0,1)$.\\

\noindent
In a recent paper \cite{Nantomah-2014}, K. Nantomah also proved the following result.
\begin{equation}\label{eqn:k-ineq1}
\frac{k^{-\frac{t}{k}}e^{-t \left( \frac{k\gamma - \gamma}{k} \right)}\Gamma(\alpha)}{\Gamma_k(\alpha)}\leq 
\frac{\Gamma(\alpha+t)}{\Gamma_k(\alpha+t)}\leq
\frac{k^{\frac{1-t}{k}}e^{(1-t) \left( \frac{k\gamma - \gamma}{k} \right)}\Gamma(\alpha+1)}{\Gamma_k(\alpha+1)}
\end{equation}
for $t\in(0,1)$, where  $\alpha$ is a positive real number.\\

\noindent
The main objective of this paper, is to establish and prove some generalizations of the inequalities  ~(\ref{eqn:p-ineq1}), ~(\ref{eqn:q-ineq1}) and ~(\ref{eqn:k-ineq1}) as previously established in \cite{Krasniqi-Shabani-2010}, 
\cite{Krasniqi-Mansour-Shabani-2010} and \cite{Nantomah-2014} respectively.


\section{Preliminaries}
\noindent
We present the following auxiliary results.

\begin{lemma}\label{lem:series-psi}
The function $\psi (t)$ as defined in ~(\ref{eqn:psi}) has the following series representation.
\begin{equation}\label{eqn:series-psi}
\psi(t)=-\gamma + (t-1) \sum_{n=0}^{\infty}\frac{1}{(1+n)(n+t)}=
-\gamma - \frac{1}{t} + \sum_{n=1}^{\infty}\frac{t}{n(n+t)}
\end{equation}
where $\gamma$ is the Euler-Mascheroni's constant.\\
\end{lemma}

\begin{proof}
See \cite{Whittaker1969}.
\end{proof}

\begin{lemma}\label{lem:series-p-psi}
The function $\psi_p(t)$ as defined in ~(\ref{eqn:p-psi}) has the following series representation.
\begin{equation}\label{eqn:series-p-psi}
\psi_p(t)=\ln p - \sum_{n=0}^{p}\frac{1}{n+t}
\end{equation}
\end{lemma}

\begin{proof}
See \cite{Krasniqi-Shabani-2010}.
\end{proof}

\begin{lemma}\label{lem:series-q-psi}
The function $\psi_q(t)$ as defined in ~(\ref{eqn:q-psi}) has the following series representation.
\begin{equation}\label{eqn:series-q-psi}
\psi_q(t)= -\ln(1-q) + \ln q \sum_{n=0}^{\infty}\frac{q^{t+n}}{1-q^{t+n}}
\end{equation}
\end{lemma}

\begin{proof}
See \cite{Krasniqi-Mansour-Shabani-2010}.
\end{proof}

\begin{lemma}\label{lem:series-k-psi}
The function $\psi_k (t)$ as defined in ~(\ref{eqn:k-psi}) also has the following series representation.
\begin{equation}\label{eqn:series-k-psi}
\psi_k(t)=\frac{\ln k-\gamma}{k}-\frac{1}{t}+\sum_{n=1}^{\infty} \frac{t}{nk(nk+t)} 
\end{equation}
\end{lemma}

\begin{proof}
See \cite{Nantomah-2014}
\end{proof}


\section{Main Results}
\noindent
We now state and prove the results of this paper.

\begin{lemma}\label{lem:psi-p-psi}
Let $a>0$, $b>0$ and  $t>1$. Then,
\begin{equation*}\label{eqn:psi-p-psi} 
a\gamma + b\ln p +a\psi(t) - b\psi_p(t)>0
\end{equation*}
\end{lemma}

\begin{proof}
Using the series representations in equations ~(\ref{eqn:series-psi}) and ~(\ref{eqn:series-p-psi}) we have,
\begin{equation*}
a\gamma + b\ln p +a\psi(t) - b\psi_p(t)= 
a(t-1)\sum_{n=0}^{\infty}\frac{1}{(1+n)(n+t)} + b\sum_{n=0}^{p}\frac{1}{(n+t)} >0
\end{equation*}
\end{proof}

\begin{lemma}\label{lem:psi-p-psi2}
Let  $a>0$, $b>0$  and  $\alpha+\beta t>1$. Then,
\begin{equation*}\label{eqn:psi-p-psi2} 
a\gamma + b\ln p +a\psi(\alpha+\beta t) - b\psi_p(\alpha+\beta t)>0
\end{equation*}
\end{lemma}
\begin{proof}
Follows directly from Lemma \ref{lem:psi-p-psi}
\end{proof}

\begin{lemma}\label{lem:psi-q-psi}
Let  $a>0$, $b>0$, $q \in(0,1)$ and  $t>1$. Then,
\begin{equation*}\label{eqn:psi-q-psi} 
a\gamma - b\ln(1-q) +a\psi(t) - b\psi_q(t)>0
\end{equation*}
\end{lemma}

\begin{proof}
Using the series representations in equations ~(\ref{eqn:series-psi}) and ~(\ref{eqn:series-q-psi}) we have,
\begin{equation*}
a\gamma - b\ln(1-q) +a\psi(t) - b\psi_q(t)= 
a(t-1)\sum_{n=0}^{\infty}\frac{1}{(1+n)(n+t)} - b\ln q\sum_{n=0}^{\infty}\frac{q^{t+n}}{1-q^{t+n}} >0
\end{equation*}
\end{proof}

\begin{lemma}\label{lem:psi-q-psi2}
Let $a>0$, $b>0$, $q \in(0,1)$ and  $\alpha+\beta t>1$. Then,
\begin{equation*}\label{eqn:psi-q-psi2} 
a\gamma - b\ln(1-q) +a\psi(\alpha+\beta t) - b\psi_q(\alpha+\beta t)>0
\end{equation*}
\end{lemma}
\begin{proof}
Follows directly from Lemma \ref{lem:psi-q-psi}
\end{proof}

\begin{lemma}\label{lem:psi-k-psi}
Let $a\geq b>0$, $k\geq1$ and  $t>0$. Then,
\begin{equation*}\label{eqn:psi-k-psi} 
\frac{ka\gamma-b\gamma}{k}+\frac{b}{k}\ln k + \frac{a-b}{t}+a\psi(t)-b\psi_k(t)\geq0
\end{equation*}
\end{lemma}

\begin{proof}
Using the series representations in equations ~(\ref{eqn:series-psi}) and ~(\ref{eqn:series-k-psi}) we have,
\begin{equation*}
\frac{ka\gamma-b\gamma}{k}+\frac{b}{k}\ln k + \frac{a-b}{t}+a\psi(t)-b\psi_k(t)= 
t \left[a\sum_{n=1}^{\infty}\frac{1}{n(n+t)} - b\sum_{n=1}^{\infty}\frac{1}{nk(nk+t)} \right]\geq0
\end{equation*}
\end{proof}

\begin{lemma}\label{lem:psi-k-psi2}
Let $a\geq b>0$, $k\geq1$ and  $\alpha+\beta t>0$. Then,
\begin{equation*}\label{eqn:psi-k-psi2} 
\frac{ka\gamma-b\gamma}{k}+\frac{b}{k}\ln k + \frac{a-b}{\alpha+\beta t}+a\psi(\alpha+\beta t)-b\psi_k(\alpha+\beta t)\geq0
\end{equation*}
\end{lemma}
\begin{proof}
Follows directly from Lemma \ref{lem:psi-k-psi}
\end{proof}

\begin{theorem}\label{thm:p-funct}
Define a function $\Omega$ by 
\begin{equation}\label{eqn:p-funct} 
\Omega(t)=\frac{p^{b\beta t}e^{a\beta \gamma t}\Gamma(\alpha+\beta t)^a}{\Gamma_p(\alpha+\beta t)^b},  \quad t\in (0,\infty), \quad p\in N
\end{equation}
where  $a$, $b$, $\alpha$, $\beta$ are  positive real numbers such that  $\alpha+\beta t>1$.
Then $\Omega$ is increasing on $t\in(0,\infty)$ and for every $t\in(0,1)$, the following inequalities are valid.
\end{theorem}

\begin{equation}\label{eqn:p-funct-ineq}
 \frac{p^{-b \beta t}e^{-a\beta \gamma t} \Gamma(\alpha)^a}{\Gamma_p(\alpha)^b}<
\frac{\Gamma(\alpha+\beta t)^a}{\Gamma_p(\alpha+\beta t)^b}<
\frac{p^{b\beta(1-t)}e^{a\beta \gamma(1-t)} \Gamma(\alpha+\beta)^a}{\Gamma_p(\alpha+\beta)^b}.
\end{equation}

\begin{proof}
Let $f(t)=\ln \Omega(t)$ for every $t\in(0,\infty)$. Then,
\begin{align*}
f(t) &=\ln \frac{p^{b\beta t}e^{a\beta \gamma t}\Gamma(\alpha+\beta t)^a}{\Gamma_p(\alpha+\beta t)^b} \\
&=b\beta t\ln p + a\beta \gamma t + a\ln \Gamma(\alpha + \beta t) - b\ln \Gamma_p(\alpha + \beta t)  \\
\intertext{Then,}
f'(t)&= a\beta \gamma  + b\beta \ln p + a\beta \psi(\alpha + \beta t) - b\beta \psi_p(\alpha + \beta t) \\
      &= \beta \left[ a\gamma  + b\ln p + a\psi(\alpha + \beta t) - b\psi_p(\alpha + \beta t) \right]>0. \quad
 \text{(by Lemma~\ref{lem:psi-p-psi2})}
\end{align*}
\end{proof}

\noindent
That implies $f$ is increasing on $t\in(0,\infty)$. Hence $\Omega$ is increasing on $t\in(0,\infty)$ and  for every $t\in(0,1)$ we have, 
\begin{equation*}
\Omega(0) < \Omega(t) < \Omega(1)  
\end{equation*}
yielding the result.

\begin{theorem}\label{thm:q-funct}
Define a function $\phi$ by 
\begin{equation}\label{eqn:q-funct} 
\phi(t)=\frac{(1-q)^{-b\beta t}e^{a\beta \gamma t}\Gamma(\alpha+\beta t)^a}{\Gamma_q(\alpha+\beta t)^b},  \quad t\in (0,\infty), \quad q\in(0,1)
\end{equation}
where  $a$, $b$, $\alpha$, $\beta$ are  positive real numbers such that  $\alpha+\beta t>1$.
Then $\phi$ is increasing on $t\in(0,\infty)$ and for every $t\in(0,1)$, the following inequalities are valid.
\end{theorem}

\begin{equation}\label{eqn:q-funct-ineq}
 \frac{(1-q)^{b \beta t}e^{-a\beta \gamma t} \Gamma(\alpha)^a}{\Gamma_q(\alpha)^b}<
\frac{\Gamma(\alpha+\beta t)^a}{\Gamma_q(\alpha+\beta t)^b}<
\frac{(1-q)^{b\beta(t-1)}e^{a\beta \gamma(1-t)} \Gamma(\alpha+\beta)^a}{\Gamma_q(\alpha+\beta)^b}.
\end{equation}

\begin{proof}
Let $g(t)=\ln \phi(t)$ for every $t\in(0,\infty)$. Then,
\begin{align*}
g(t) &=\ln \frac{(1-q)^{-b\beta t}e^{a\beta \gamma t}\Gamma(\alpha+\beta t)^a}{\Gamma_q(\alpha+\beta t)^b} \\
&= -b\beta t\ln(1-q) + a\beta \gamma t + a\ln \Gamma(\alpha + \beta t) - b\ln \Gamma_q(\alpha + \beta t)  \\
\intertext{Then,}
g'(t)&= -b\beta \ln(1-q) + a\beta \gamma + a\beta \psi(\alpha + \beta t) - b\beta \psi_q(\alpha + \beta t) \\
      &= \beta \left[ a\gamma  - b\ln(1-q) + a\psi(\alpha + \beta t) - b\psi_q(\alpha + \beta t) \right]>0. \quad
 \text{(by Lemma~\ref{lem:psi-q-psi2})}
\end{align*}
\end{proof}

\noindent
That implies $g$ is increasing on $t\in(0,\infty)$. Hence $\phi$ is increasing on $t\in(0,\infty)$ and  for every $t\in(0,1)$ we have, 
\begin{equation*}
\phi(0) < \phi(t) < \phi(1)   
\end{equation*}
yielding the result.
\begin{theorem}\label{thm:k-funct}
Define a function $\theta$ by 
\begin{equation}\label{eqn:k-funct} 
\theta(t)=\frac{(\alpha+\beta t)^{(a-b)} e^{t(\frac{ka\beta \gamma - b\beta \gamma}{k})}\Gamma(\alpha+\beta t)^a}{k^{-\frac{b\beta t}{k}} \Gamma_k(\alpha+\beta t)^b},  \quad t\in (0,\infty), \quad k\geq1
\end{equation}
where  $a$, $b$, $\alpha$, $\beta$ are  positive real numbers such that $a\geq b$.
Then $\theta$ is increasing on $t\in(0,\infty)$ and for every $t\in(0,1)$, the following inequalities are valid.
\end{theorem}

\begin{equation}\label{eqn:k-funct-ineq}
\frac{\alpha^{(a-b)}e^{-t(\frac{ka\beta \gamma - b\beta \gamma}{k})} \Gamma(\alpha)^a}{(\alpha+\beta t)^{(a-b)}k^{\frac{b\beta t}{k}}\Gamma_k(\alpha)^b}\leq
\frac{\Gamma(\alpha+\beta t)^a}{\Gamma_k(\alpha+\beta t)^b}\leq
\frac{(\alpha+\beta)^{(a-b)}e^{(1-t)(\frac{ka\beta \gamma - b\beta \gamma}{k})} \Gamma(\alpha+\beta)^a}{(\alpha+\beta t)^{(a-b)} k^{\frac{b\beta}{k}(t-1)} \Gamma_k(\alpha+\beta)^b}.
\end{equation}


\begin{proof}
Let $h(t)=\ln \theta(t)$ for every $t\in(0,\infty)$. Then,
\begin{align*}
h(t) &=\ln \frac{(\alpha+\beta t)^{(a-b)} e^{t(\frac{ka\beta \gamma - b\beta \gamma}{k})}\Gamma(\alpha+\beta t)^a}{k^{-\frac{b\beta t}{k}} \Gamma_k(\alpha+\beta t)^b}\\
&= (a-b)\ln(\alpha+\beta t) + \frac{b\beta t}{k}\ln k +t(\frac{ka\beta \gamma - b\beta \gamma}{k}) \\
&\quad + a\ln \Gamma(\alpha + \beta t) - b\ln \Gamma_k(\alpha + \beta t)  \\
\intertext{Then,}
h'(t)&=\frac{ka\beta \gamma - b\beta \gamma}{k}+\frac{b\beta}{k}\ln k +\beta \frac{a-b}{\alpha+\beta t}+ a\beta \psi(\alpha + \beta t) - b\beta \psi_k(\alpha + \beta t) \\
      &= \beta \left[ \frac{ka \gamma - b\gamma}{k}+\frac{b}{k}\ln k +\frac{a-b}{\alpha+\beta t}+ a\psi(\alpha + \beta t) - b\psi_k(\alpha + \beta t)\right] \geq0 
\end{align*}
\end{proof}

\noindent
That is as a result of Lemma~\ref{lem:psi-k-psi2}. That implies $h$ is increasing on $t\in(0,\infty)$. Hence $\theta$ is increasing on $t\in(0,\infty)$ and  for every $t\in(0,1)$ we have, 
\begin{equation*}
\theta(0) \leq \theta(t) \leq \theta(1) 
\end{equation*}
yielding the result.

\section{Concluding Remarks}

\noindent
We dedicate this section to some remarks concerning our results.
\begin{remark}
If in Theorem \ref{thm:p-funct} we set $a=b=\beta = 1$,  then the inequalities in ~(\ref{eqn:p-ineq1}) are restored.
\end{remark}

\begin{remark}
If in Theorem \ref{thm:q-funct} we set $a=b=\beta = 1$,  then the inequalities in ~(\ref{eqn:q-ineq1}) are restored.
\end{remark}

\begin{remark}
If in Theorem \ref{thm:k-funct} we set $a=b=\beta = 1$,  then the inequalities in ~(\ref{eqn:k-ineq1}) are restored.
\end{remark}

\noindent
With the foregoing Remarks, the results of \cite{Krasniqi-Shabani-2010}, \cite{Krasniqi-Mansour-Shabani-2010} and \cite{Nantomah-2014} have been generalized.\\


\bibliographystyle{plain}


\end{document}